\documentclass{article}
\usepackage{amsfonts}
\usepackage{amsmath, amssymb}
\usepackage{amsthm}

\theoremstyle{plain}
\newtheorem{theorem}{\indent\sc Theorem}[section]
\newtheorem{lemma}[theorem]{\indent\sc Lemma}
\newtheorem{corollary}[theorem]{\indent\sc Corollary}
\newtheorem{proposition}[theorem]{\indent\sc Proposition}

\theoremstyle{definition}
\newtheorem{definition}[theorem]{\indent\sc Definition}

\newtheorem{example}[theorem]{\indent\sc Example}

\newcommand\on{\operatorname}

\newcommand\Span{\on{Span}}


\begin{document}

\title{Some properties of pointwise $k$-slant submanifolds of K\"{a}hler manifolds}
\author{Adara M. Blaga and Dan Radu La\c tcu}
\date{}
\maketitle

\begin{abstract}
We study some properties of pointwise $k$-slant submanifolds of almost Hermitian manifolds with a special view towards
K\"{a}hler manifolds. In particular, we characterize the integrability of the component distributions, treating also the totally geodesic case.
\end{abstract}

\markboth{{\small\it {\hspace{0.5cm} Some properties of pointwise $k$-slant submanifolds of K\"{a}hler manifolds}}}{\small\it{Some properties of pointwise $k$-slant submanifolds of K\"{a}hler manifolds \hspace{0.5cm}}}

\footnote{%2010 MSC numbers
}
\footnote{%key words and phrases
}

\bigskip

\section{Introduction and Preliminaries}

\textit{Slant} and \textit{pointwise slant submanifolds}, introduced by Chen \hspace{-2pt}\cite{chen} and Etayo \hspace{-2pt}\cite{etayo} (see also \cite{chen5}), respectively,
have been intensively investigated in different geometries. Recently, La\c tcu \hspace{-2pt}\cite{latcu} defined the more general notions of \textit{$k$-slant} and \textit{pointwise \mbox{$k$-slant} submanifold} of an almost product Riemannian, an almost Hermitian, and an almost contact or paracontact metric manifold,
involving the decomposition of the tangent bundle of a submanifold into a sum of orthogonal slant or pointwise slant distributions.
It's to be mentioned that Ronsse \hspace{-2pt}\cite{ro} and Chen \hspace{-2pt}\cite{geo,chen}
considered, in the almost Hermitian case, the orthogonal decomposition of
the tangent space in a point of a submanifold
into the direct sum of the eigenspaces corresponding to the square of the tangential component of the structural tensor field.
Accordingly, the submanifold was called \hspace{-2pt}\cite{ro} a \textit{generic submanifold}, or, under some restrictions, a \textit{skew CR submanifold}.

In this paper, we focus on some properties of pointwise $k$-slant submanifolds of almost Hermitian manifolds, with a special view toward the K\"{a}hler case. More precisely, we characterize the integrability of the component distributions, and we obtain some properties of such submanifolds with parallel tensor fields, discussing also the totally geodesic case.

\bigskip

Let $(\bar {M}, g)$ be a Riemannian manifold, and let $\varphi$ be a $(1,1)$-tensor field on $\bar {M}$. We recall that $(\bar M,\varphi,g)$ is said to be an \textit{almost Hermitian manifold} if
$$\varphi^{2}=-I \ \ \textrm{and} \ \ g(\varphi \cdot,\varphi \cdot)=g,$$
which further gives
$$g(\varphi \cdot,\cdot) =-g(\cdot,\varphi \cdot).$$

If the structural endomorphism $\varphi$ satisfies $\bar \nabla \varphi=0$, where $\bar \nabla$ is the Levi-Civita connection of $g$, then $\bar M$ is called a \textit{K\"{a}hler manifold}.

\bigskip

For a submanifold $M$ of an almost Hermitian manifold $(\bar M,\varphi,g)$ defined by an injective immersion, we will denote the induced metric on $M$ also with $g$ and by $\nabla$ the Levi-Civita connection on $M$. The Gauss and Weingarten equations are:
$$\bar \nabla_XY=\nabla_XY+h(X,Y) \ \ \textrm{and} \ \ \bar \nabla_XV=-A_VX+\nabla^{\bot}_XV$$
for all $X,Y\in \Gamma(TM)$ and $V\in \Gamma(T^{\bot}M)$, where $h$ is the second fundamental form and $A$ is the shape operator, related by $g(h(X,Y),V)=g(A_VX,Y)$.

We have the orthogonal decomposition
$$T\bar M=TM\oplus T^{\bot}M,$$
and, for all $X\in \Gamma(TM)$ and $V\in \Gamma(T^{\bot}M)$, we will write:
$$\varphi X=TX+NX \ \ \textrm{and} \ \ \varphi V=tV+nV,$$
where $TX$, $NX$ and $tV$, $nV$ stand for the tangent and the normal component of $\varphi X$ and $\varphi V$, respectively.

\section{Pointwise $k$-slant submanifolds of \\ almost Hermitian manifolds}

We recall that a distribution $\mathcal D\subseteq TM$ is called a \textit{pointwise slant distribution} if, at each point $p \in M$, the angle $\theta(p)$ between $\varphi X_p$ and $\mathcal D_p$ is nonzero and independent of the choice of the tangent vector $X_p \in \mathcal D_{p}\verb=\=\{0\}$ (but it depends on $p \in M$). In this case, the function $\theta$ is called the \textit{slant function}.

\begin{definition} {\hspace{-2pt}\rm\cite{latcu}} \label{d1}
A submanifold $M$, defined by an injective immersion, of an almost Hermitian manifold $(\bar M,\varphi, g)$
is said to be a \textit{pointwise $k$-slant submanifold of $\bar M$} ($k\in \mathbb N^*$) if there exist some orthogonal smooth regular distributions, $\mathcal D_{0},\mathcal D_1,\dots,\mathcal D_{k}$, satisfying:

(i) $TM=\mathcal D_{0}\oplus \mathcal D_1 \oplus \dots \oplus \mathcal D_{k}$;

(ii) $T\mathcal (D_i)\subseteq \mathcal D_i$ for any $i\in \{1,\dots, k\}$;

(iii) $\mathcal D_0$ is invariant (or even trivial) and $\mathcal D_{i}$, $i \in \{1,\dots, k\}$, are nontrivial, pointwise slant distributions with their slant functions $\theta_i$,
$\theta_i(p)\in (0,\frac{\pi}{2}]$ for $p\in M$ and $i \in \{1,\dots, k\}$, which are pointwise distinct (i.e., $\theta_i(p)\neq \theta_j(p)$ for all $p\in M$ and $i\neq j$).
\end{definition}

Conventionally, we will denote by $\theta_0$ the null angle, i.e., the "slant" angle of the invariant distribution $\mathcal D_0$ (if $\mathcal D_0$ is not trivial).

\smallskip

We notice \hspace{-2pt}\cite{latcu} that the condition (ii) from the Definition \ref{d1} is equivalent to: $\varphi(\mathcal D_i)\bot \mathcal D_j$ for all $i\neq j$, $i,j \in \{1,\dots, k\}$.

\smallskip

The slant functions $\theta_i$ are continuous (even smooth, under a certain assumption) \hspace{-2pt}\cite{latcu2}, and, for all $X\in \nolinebreak\Gamma(\mathcal D_i)\setminus \{0\}$ and $p\in M$, the angle $\theta_i(p)$ between $\varphi X_p$ and $T_pM$ coincides with the angle between $\varphi X_p$ and $(\mathcal D_i)_p$, and it satisfies
$$\cos\theta_i (p)\cdot \Vert\varphi X_p\Vert= \Vert TX_p \Vert.$$

If $\theta_i$ is constant for all $i \in \{1,\dots, k\}$, then the submanifold $M$ is called a \textit{$k$-slant submanifold} \hspace{-1pt}\cite{latcu}, so all the results for pointwise $k$-slant submanifolds are also valid for $k$-slant submanifolds.

\smallskip

Now, we will construct an example of a pointwise $k$-slant submanifold and one of a $k$-slant submanifold of a K\"{a}hler manifold.

\begin{example}\label{ex2}
Let us consider the K\"{a}hler manifold $\left(\mathbb{R}^{6k},\varphi,\left\langle \cdot,\cdot\right\rangle \right)$, $k\geq 2$,
with the standard Euclidean metric $\left\langle \cdot,\cdot\right\rangle$ and $\varphi$ given by
$$\varphi \left( \frac{\partial }{\partial u_{i}}\right) =-\frac{\partial }{%
\partial v_{i}},\quad \varphi \left( \frac{\partial }{\partial v_{i}}\right) =%
\frac{\partial }{\partial u_{i}},$$
where $(u_{1},v_{1},\dots, u_{3k},v_{3k})$ are the canonical coordinates in $\mathbb{R}^{6k}$.
We consider the submanifold $M$ of $\mathbb{R}^{6k}$ defined by the immersion
$$f :\{z=(x_{1},x_{2},y_{1},\dots,y_{2k-1})\in \mathbb{R}^{2k+1}: \Vert z\Vert<1, x_1>0, x_2>0\}\rightarrow \mathbb{R}^{6k},$$
\begin{align*}
f(x_{1},x_{2},y_{1},\dots,y_{2k-1}) &:=
\Big(x_{1}\cos y_{1}, x_{2}\cos y_{1}, x_{1}\sin y_{1}, x_{2}\sin y_{1}, x_{1}, x_{2},\\
&\hspace{24pt}y_2, \frac{1}{2}y_2^2, y_2+y_3,y_2-y_3, y_3, \frac{1}{2}y_{3}^2, \dots,\\
&\hspace{24pt}(k-1)y_{2k-2}, \frac{1}{2}y_{2k-2}^2, y_{2k-2}+y_{2k-1},\\
&\hspace{24pt}y_{2k-2}-y_{2k-1},(k-1)y_{2k-1}, \frac{1}{2}y_{2k-1}^2\Big).
\end{align*}%

Then, $TM$ is spanned by
\begin{align*}
X_{1} &=\cos y_{1}\frac{\partial }{\partial u_{1}}+\sin y_{1}\frac{\partial
}{\partial u_{2}}+\frac{\partial }{\partial u_{3}}, \\
X_{2} &=\cos y_{1}\frac{\partial }{\partial v_{1}}+\sin y_{1}\frac{\partial
}{\partial v_{2}}+\frac{\partial }{\partial v_{3}}, \\
X_{3} &=-x_{1}\sin y_{1}\frac{\partial }{\partial u_{1}}-x_{2}\sin y_{1}%
\frac{\partial }{\partial v_{1}}+x_{1}\cos y_{1}\frac{\partial }{\partial
u_{2}}+x_{2}\cos y_{1}\frac{\partial }{\partial v_{2}}, \\
X_{2i} &=(i-1)\frac{\partial }{\partial u_{3i-2}}+y_{2i-2}
\frac{\partial }{\partial v_{3i-2}}+\frac{\partial }{\partial
u_{3i-1}}+\frac{\partial }{\partial v_{3i-1}}, \\
X_{2i+1} &=\frac{\partial }{\partial u_{3i-1}}-\frac{\partial }{\partial v_{3i-1}}+(i-1)\frac{\partial }{\partial
u_{3i}}+y_{2i-1}\frac{\partial }{\partial v_{3i}}
\end{align*}%
for $i\in \{2,3,\dots, k\}$. We notice that $X_1, X_2,\dots, X_{2k+1}$ are mutually orthogonal.

Applying $\varphi $ to the base vector fields of $TM$, we get%
\begin{align*}
\varphi X_{1} &=-\cos y_{1}\frac{\partial }{\partial v_{1}}-\sin y_{1}\frac{\partial
}{\partial v_{2}}-\frac{\partial }{\partial v_{3}}, \\
\varphi X_{2} &=\cos y_{1}\frac{\partial }{\partial u_{1}}+\sin y_{1}\frac{\partial
}{\partial u_{2}}+\frac{\partial }{\partial u_{3}}, \\
\varphi X_{3} &=-x_{2}\sin y_{1}%
\frac{\partial }{\partial u_{1}}+x_{1}\sin y_{1}\frac{\partial }{\partial v_{1}}+x_{2}\cos y_{1}\frac{\partial }{\partial u_{2}}-x_{1}\cos y_{1}\frac{\partial }{\partial
v_{2}}, \\
\varphi X_{2i} &=y_{2i-2}
\frac{\partial }{\partial u_{3i-2}}-(i-1)\frac{\partial }{\partial v_{3i-2}}+\frac{\partial }{\partial
u_{3i-1}}-\frac{\partial }{\partial v_{3i-1}}, \\
\varphi X_{2i+1} &=-\frac{\partial }{\partial u_{3i-1}}-\frac{\partial }{\partial v_{3i-1}}+y_{2i-1}\frac{\partial }{\partial u_{3i}}-(i-1)\frac{\partial }{\partial
v_{3i}}
\end{align*}
for $i\in \{2,3,\dots, k\}$. We immediately obtain $\varphi X_{1}=-X_{2}$ and $\varphi X_{2}=X_{1}$; therefore, the distribution $\mathcal D_0=\Span \{X_{1},X_{2}\}$ is an invariant distribution.
Also, we have
\begin{align*}\frac{\left\vert \left\langle \varphi X_{2i},X_{2i+1}\right\rangle \right\vert }{%
\left\Vert \varphi X_{2i}\right\Vert \cdot\left\Vert X_{2i+1}\right\Vert } &=\frac{\left\vert \left\langle \varphi X_{2i+1},X_{2i}\right\rangle \right\vert }{\left\Vert \varphi X_{2i+1}\right\Vert \cdot\left\Vert X_{2i}\right\Vert }\\
&=\frac{2}{\sqrt{\left(2+(i-1)^2+y_{2i-2}^2\right)\left(2+(i-1)^2+y_{2i-1}^2\right)}};
\end{align*}
hence, since $X_{2i}$ and $X_{2i+1}$ are orthogonal, the distribution $\mathcal D_{i}=\Span\{X_{2i},X_{2i+1}\}$, $i\in \{2,3,\dots, k\}$, is a pointwise slant distribution with the slant function defined by $$\theta_{i}(f(x_{1},x_{2},y_{1},\dots,y_{2k-1})) =\arccos \frac{2}{\sqrt{\left(2+(i-1)^2+y_{2i-2}^2\right)\left(2+(i-1)^2+y_{2i-1}^2\right)}}.$$

By a direct computation, we find
$$\left\langle \varphi X_{3},X_{1}\right\rangle =\left\langle \varphi X_{3},X_{2}\right\rangle=\left\langle \varphi X_{3},X_{3}\right\rangle=\left\langle \varphi X_{3},X_{2i}\right\rangle=\left\langle \varphi X_{3},X_{2i+1}\right\rangle =0$$
for any $i\in \{2, 3,\dots, k\}$,
and so the distribution $\mathcal D_1=\Span\{X_{3}\}$ is an anti-invariant distribution.

Therefore, we can conclude that $M$ is a pointwise $k$-slant submanifold of the K\"{a}hler manifold $\left( \mathbb{R}^{6k},\varphi,\left\langle \cdot,\cdot\right\rangle \right)$.
\end{example}

\begin{example}\label{ex2}
The submanifold obtained by replacing $\frac{1}{2}y_j^2$ with $y_j$ for any $j\in \{2,3,\dots, 2k-1\}$, in the expression of the immersion from the previous example is a $k$-slant submanifold of the standard K\"{a}hler manifold $\left(\mathbb{R}^{6k},\varphi,\left\langle \cdot,\cdot\right\rangle \right)$, with the slant angles
$$\theta_{i} =\arccos \frac{2}{3+(i-1)^2}$$ for the slant distributions $\mathcal D_i$, $i\in \{2,3,\dots, k\}$, respectively, and $\theta_{1} = \frac{\pi}{2}$ for $\mathcal D_1$.
\end{example}

\bigskip

If $M$ is a pointwise $k$-slant submanifold of an almost Hermitian manifold $(\bar M,\varphi,g)$, then we have the following decompositions \hspace{-1pt}\cite{latcu} of the tangent and normal bundles of $M$:
$$TM=\oplus_{i=0}^k\mathcal D_i, \ \ T^{\bot}M=\oplus_{i=1}^kN(\mathcal D_i)\oplus H,$$ where $\varphi(H)=H$.
Let $P_i$ be the projection from $TM$ onto $\mathcal D_i$, $i \in \{0,\dots, k\}$, $Q_i$ be the projection from $T^{\bot}M$ onto $N(\mathcal D_i)$, $i \in \{1,\dots, k\}$, and
$Q_0$ be the projection from $T^{\bot}M$ onto $H$. Then, any $X\in \Gamma(TM)$ and $V\in  \Gamma(T^{\bot}M)$ can be written as:
$$X=\sum_{i=0}^kP_iX, \ \ V=\sum_{i=0}^kQ_iV.$$

By a direct computation, we immediately obtain (see also \cite{latcu}) the following

\begin{lemma} \label{le6}
If $M$ is a pointwise $k$-slant submanifold of an almost Hermitian manifold $(\bar M,\varphi,g)$, then:
\begin{align*}
(i)\ &
g(TX,Y)=-g(X,TY), \ \ g(NX,V)=-g(X,tV), \ \ g(nV,W)=-g(V,nW)
\intertext{for all $X,Y\in \Gamma (TM)$ and $V,W\in \Gamma (T^{\bot}M)$;}
(ii)\ & T^2=-\sum_{i=0}^k\cos ^{2}\theta_{i} \cdot P_i, \ \ n^2=-\sum_{i=0}^k\cos ^{2}\theta_{i} \cdot Q_i;\\
(iii)\ &
g(TX,TY)=\sum_{i=0}^k\cos^2 \theta_i \cdot g(P_iX,P_iY),
\\&  g(NX,NY)=\sum_{i=1}^k\sin^2 \theta_i \cdot g(P_iX,P_iY)
\intertext{for all $X,Y\in \Gamma (TM)$.}
\end{align*}
\end{lemma}

\section{On the integrability of the component \\distributions}

For a pointwise $k$-slant submanifold $M$ of an almost Hermitian manifold $(\bar M,\varphi,g)$, we will denote $TM=\oplus_{i=0}^k\mathcal D_{i}$.
Similarly to the almost contact metric case \hspace{-1pt}\cite{blla}, the integrability of the component distributions $\mathcal D_i$, $i \in \{0,\dots, k\}$, can be characterized in the K\"{a}hler case as follows.

We recall that a distribution $\mathcal D$ is called \textit{integrable} if $[X,Y]\in \Gamma(\mathcal D)$ for all $X,Y\in \Gamma (\mathcal D)$, and \textit{completely integrable} if $\nabla_XY\in \Gamma(\mathcal D)$ for all $X,Y\in\Gamma(\mathcal D)$.

\begin{theorem}
If $M$ is a pointwise $k$-slant submanifold of a K\"{a}hler manifold $(\bar M,\varphi,g)$, then:

(i) for $i \in \{0,\dots, k\}$, $\mathcal D_i$ is an integrable distribution if and only if
$$g(X,\nabla_YZ)=g(Y,\nabla_XZ)$$ for all $X,Y\in \Gamma(\mathcal D_i)$ and $Z\in \Gamma(\mathcal D_j)$, $j \in \{0,\dots, k\}$ with $j\neq i$;

(ii) $\mathcal D_0$ is an integrable distribution if and only if
$$h(X,TY)=h(TX,Y)$$
for all $X,Y\in \Gamma (\mathcal D_0)$;

(iii) for $i \in \{0,\dots, k\}$ with $\theta_j(p)\neq\frac{\pi}{2}$ for all $p\in M$ and $j\neq i$, $j\in\{0,\dots,k\}$,
$\mathcal D_{i}$ is an integrable distribution if and only if
$$\nabla _{X}TY-\nabla _{Y}TX+A_{NX}Y-A_{NY}X\in \Gamma(\mathcal D_i)$$
for all $X,Y\in \Gamma (\mathcal D_{i})$.
\end{theorem}
\begin{proof}
We have
$$
\bar{\nabla }_{X}\varphi Y=\varphi (\bar{\nabla }_{X}Y)
$$
for all $X,$ $Y\in \Gamma (TM)$, and, using Gauss and Weingarten equations,
we get
\begin{align*}
\nabla _{X}TY+h(X,TY)&=T(\nabla_XY)+N(\nabla_XY)+A_{NY}X
-\nabla_X^{\bot}NY\\
&\hspace{12pt}+th(X,Y)+nh(X,Y)\\
&=\sum_{i=0}^kTP_i(\nabla_XY)+\sum_{i=1}^kNP_i(\nabla_XY)
+A_{NY}X-\nabla_X^{\bot}NY\\
&\hspace{12pt}+th(X,Y)+nh(X,Y).
\end{align*}

Identifying the tangent and the normal components in the previous relation, we obtain:
\begin{align*}
\nabla _{X}TY&=\sum_{i=0}^kTP_i(\nabla_XY)+A_{NY}X+th(X,Y),\\
h(X,TY)&=\sum_{i=1}^kNP_i(\nabla_XY)-\nabla_X^{\bot}NY+nh(X,Y)
\end{align*}
for all $X,$ $Y\in \Gamma (TM).$

The distribution $\mathcal D_i$ is integrable if and only if $g([X,Y],Z)=0$ for all $X,Y\in \Gamma(\mathcal D_i)$ and $Z\in \Gamma(\mathcal D_j)$, $j \in \{0,\dots, k\}$ with $j\neq i$.
Since $g([X,Y],Z)=g(X,\nabla_YZ)-g(Y,\nabla_XZ)$, we get (i).

Let $X,Y\in \Gamma (\mathcal D_0)$. Then, $NX=NY=0$, and we obtain
$$h(X,TY)-h(TX,Y)=\sum_{i=1}^kNP_i[X,Y].$$

If the distribution $\mathcal D_0$ is integrable, then $[X,Y]\in \Gamma(\mathcal D_0)$; therefore,
\linebreak
$P_i[X,Y]=0$ for all $i \in \{1,\dots, k\}$, hence the conclusion.
Conversely, if
\linebreak
$h(X,TY)=h(TX,Y)$ for all $X,Y\in \Gamma (\mathcal D_0)$, then $\sum_{i=1}^kNP_i[X,Y]=0$, hence $P_i[X,Y]=0$ for all
$i \in \{1,\dots, k\}$, and we get (ii).

Let $X,Y\in \Gamma (\mathcal D_i)$. Then, we obtain
$$T[X,Y]=T(\nabla_XY)-T(\nabla_YX)=\nabla _{X}TY-\nabla_{Y}TX+A_{NX}Y-A_{NY}X.$$
If the distribution $\mathcal D_{i}$ is integrable, then $[X,Y]\in \Gamma(\mathcal D_i)$; therefore, \linebreak
$T[X,Y]\in \Gamma(\mathcal D_i)$, hence the conclusion.
Conversely, if $\nabla _{X}TY-\nabla _{Y}TX+A_{NX}Y-A_{NY}X\in \Gamma(\mathcal D_i)$, then $T[X,Y]\in \Gamma(\mathcal D_i)$, and, since $\theta_j(p)\neq\frac{\pi}{2}$ for all $p\in M$ and $j\neq i$, we get (iii).
\end{proof}

In particular, for a totally geodesic submanifold (i.e., for $h=0$), we deduce

\begin{corollary}
If $M$ is a totally geodesic pointwise $k$-slant submanifold of a K\"{a}hler manifold $(\bar M,\varphi,g)$, then:

(i) $\mathcal D_0$ is an integrable distribution;

(ii) for $i \in \{0,\dots, k\}$ with $\theta_j(p)\neq\frac{\pi}{2}$ for all $p\in M$ and $j\neq i$, $j\in\{0,\dots,k\}$,
$\mathcal D_{i}$ is an integrable distribution if and only if
$$\nabla _{X}TY-\nabla _{Y}TX\in \Gamma(\mathcal D_i)$$
for all $X,Y\in \Gamma (\mathcal D_{i})$.
\end{corollary}

\smallskip

For a submanifold $M$ of an almost Hermitian manifold $(\bar M,\varphi,g)$ defined by an injective immersion,
by using the Gauss and Weingarten equations,
for any $X,Y\in \Gamma(TM)$ and $V\in \Gamma(T^{\bot}M)$, we obtain:
\begin{align*}
(\bar \nabla _X\varphi)Y&=\nabla_XTY-T(\nabla_XY)-A_{NY}X-th(X,Y)\\
&\hspace{12pt}+\nabla^{\bot}_XNY-N(\nabla_XY)+h(X,TY)-nh(X,Y), \\
(\bar \nabla _X\varphi)V&=\nabla_XtV-t(\nabla^{\bot}_XV)-A_{nV}X+T(A_VX)\\
&\hspace{12pt}+\nabla^{\bot}_XnV-n(\nabla^{\bot}_XV)+h(X,tV)+N(A_VX). \notag
\end{align*}
Denoting:
\begin{align*}
(\nabla_XT)Y&:=\nabla_XTY-T(\nabla_XY), \hspace{-20pt} &(\nabla_XN)Y&:=\nabla^{\bot}_XNY-N(\nabla_XY),\\
(\nabla_Xt)V&:=\nabla_XtV-t(\nabla^{\bot}_XV), \hspace{-20pt} &(\nabla_Xn)V&:=\nabla^{\bot}_XnV-n(\nabla^{\bot}_XV)
\end{align*}
for all $X,Y\in \Gamma(TM)$ and $V\in \Gamma(T^{\bot}M)$, by identifying the tangent and the normal components in the K\"{a}hler case, we get the following

\begin{lemma} \label{l}
If $(\bar M,\varphi,g)$ is a K\"{a}hler manifold, then, for all $X,Y\in \Gamma(TM)$ and $V\in \Gamma(T^{\bot}M)$,
we have:

(i) $(\nabla_XT)Y=A_{NY}X+th(X,Y)$;

(ii) $(\nabla_XN)Y=-h(X,TY)+nh(X,Y)$;

(iii) $(\nabla_Xt)V=A_{nV}X-T(A_VX)$;

(iv) $(\nabla_Xn)V=-h(X,tV)-N(A_VX)$.

Moreover, if $M$ is totally geodesic, then we get: ${\nabla}T=0$, $\nabla N=0$, $\nabla t=0$, and $\nabla n=0$.
\end{lemma}

\smallskip

We recall that a $(1,1)$-tensor field $J$ on $M$ is called \textit{parallel} if $(\nabla_X J)Y=0$ for all $X,Y\in \Gamma(TM)$.
We will characterize the property of $T$ and $N$ to be parallel tensor fields as follows.

\begin{proposition}\label{k}
If $(\bar M,\varphi,g)$ is a K\"{a}hler manifold, then:

(i) $\nabla T=0$ is equivalent to $A_{NY}X=A_{NX}Y$ for all $X,Y\in \Gamma(TM)$;

(ii) $\nabla N=0$ is equivalent to any of the following assertions:

\quad (1) $h(TX,Y)=h(X,TY)$ for all $X,Y\in \Gamma(TM)$;

\quad (2) $T(A_VX)=-A_V(TX)$ for all $X\in \Gamma(TM)$ and $V\in \Gamma(T^{\bot}M)$;

\quad (3) $A_{nV}X=-A_V(TX)$ for all $X\in \Gamma(TM)$ and $V\in\Gamma(T^{\bot}M)$.
\end{proposition}
\begin{proof}
Since $T$ is skew-symmetric, the condition $\nabla T=0$ is equivalent to $(\nabla_XT)Y=(\nabla_YT)X$ for all $X,Y\in \nolinebreak\Gamma(TM)$, which, by means of Lemma \ref{l} (i), is equivalent to $A_{NY}X=A_{NX}Y$ for all $X,Y\in\Gamma(TM)$, and we get (i).

We shall prove now:
$$\nabla N=0 \Longrightarrow \ (1) \ \Longrightarrow \ (2) \ \Longrightarrow \nabla N=0 \ \ \textrm{and} \ \ \nabla N=0 \Longrightarrow \ (3) \ \Longrightarrow \ (1).$$

If $\nabla N=0$, we immediately get (1) from Lemma \ref{l} (ii). So,
\begin{align*}g(A_V(TX),Y)&=g(h(TX,Y),V)=g(h(X,TY),V)\\
&=g(A_VX, TY)=-g(T(A_{V}X),Y)
\end{align*}
for all $X,Y\in\Gamma(TM)$ and $V\in\Gamma(T^{\bot}M)$,
and we obtain (2).
Further,
\begin{align*}g(h(X,TY),V)&=g(A_{V}TY,X)=-g(T(A_VY), X)\\
&=g(A_VY,TX)=g(h(TX,Y),V)
\end{align*}
for all $X,Y\in \Gamma(TM)$ and $V\in\Gamma(T^{\bot}M)$; hence, from Lemma \ref{l} (ii), we obtain $(\nabla_XN)Y=(\nabla_YN)X$ for all $X,Y\in \Gamma(TM)$, which,
since $N$ is skew-symmetric, is equivalent to $\nabla N=0$.

Also, $\nabla N=0$ is equivalent to $h(X,TY)=nh(X,Y)$ for all $X,Y\in\Gamma(TM)$; hence,
\begin{align*}g(A_V(TY),X)&=g(h(X,TY),V)=g(nh(X,Y),V)\\
&=-g(h(X,Y),nV)=-g(A_{nV}Y,X)
\end{align*}
for all $X,Y\in\Gamma(TM)$ and $V\in\Gamma(T^{\bot}M)$, and we deduce (3).
Further,
\begin{align*}g(h(TX,Y),V)&=g(A_V(TX),Y)=-g(A_{nV}X,Y)=-g(h(X,Y),nV)\\
&=-g(A_{nV}Y,X)=g(A_V(TY),X)=g(h(X,TY),V)
\end{align*}
for all $X,Y\in\Gamma(TM)$ and $V\in\Gamma(T^{\bot}M)$, hence (1).
\end{proof}

\begin{theorem}
Let $M$ be a pointwise $k$-slant submanifold of a K\"{a}hler manifold $(\bar M,\varphi,g)$.
If $\nabla T=0$, then:

(i) $\mathcal D_0$ and $\oplus_{i=1}^k \mathcal D_i$ are completely integrable distributions;

(ii) for $i\in \{0,\dots,k\}$ with $\theta_j(p)\neq \frac{\pi}{2}$ for all $p\in M$ and $j\neq i$, $j\in \{0,\dots,k\}$, $\mathcal D_i$ is an integrable distribution if and only if, for all $X,Y\in \Gamma(\mathcal D_i)$,
$$\nabla_XTY-\nabla_YTX\in \Gamma(\mathcal D_i);$$

(iii) either $M$ is a $(\mathcal D_0,\mathcal D_0)$-totally geodesic submanifold of $\bar M$ (i.e.,
\linebreak
 $h(X,Y)=0$ for all $X,Y\in \Gamma(\mathcal D_0)$), or $(-1)$ is an eigenvalue of $n^2$, and, for $X,Y\in \Gamma(\mathcal D_0)$, any nonzero $h(X,Y)$ is an eigenvector for it.
\end{theorem}
\begin{proof}
For all $X\in \Gamma(TM)$ and $Y\in \Gamma(\mathcal D_0)$, we have
$$0=(\bar \nabla_X \varphi)Y=-N(\nabla_XY)+h(X,TY)-\varphi h(X,Y)$$
from Gauss and Weingarten equations, which, for all $X\in \Gamma(TM)$ and
\linebreak
$Y\in \Gamma(\mathcal D_0)$, implies
$$\sum_{i=1}^kNP_i(\nabla_XY)=h(X,TY)-\varphi h(X,Y).$$

Since $\nabla T=0$, we have $th(X,Y)=0$ for all $X\in \Gamma(TM)$ and $Y\in \Gamma(\mathcal D_0)$,
so
$$g(\varphi h(X,Y),Z)=-g(h(X,Y),\varphi Z)=-g(h(X,Y),NZ)=g(th(X,Y),Z)=0$$
for all $X,Z\in \Gamma(TM)$ and $Y\in \Gamma(\mathcal D_0)$. On the other hand,
for $Y\in \Gamma(\mathcal D_0)$, we have $TY\in \Gamma(\mathcal D_0)$, which gives
$g(h(X,TY),NZ)=0$, and we obtain
\begin{align*}\Big\Vert N\big(\sum_{i=1}^kP_i(\nabla_XY)\big)\Big\Vert^2&=g\Big(h(X,TY),N\big(\sum_{i=1}^kP_i(\nabla_XY)\big)\Big)\\
&\hspace{12pt}-g\Big(\varphi h(X,Y),\sum_{i=1}^kNP_i(\nabla_XY)\Big) =0
\end{align*}
for all $X\in \Gamma(TM)$ and $Y\in \Gamma(\mathcal D_0)$.
Since
$$g(NX,NX)=\sum_{i=1}^k\sin^2 \theta_i \cdot g(P_iX,P_iX)$$
for all $X\in \Gamma(TM)$, we have
$$\Big\Vert N\big(\sum_{i=1}^kP_i(\nabla_XY)\big)\Big\Vert^2=\sum_{i=1}^k\sin^2 \theta_i \cdot\Big\Vert P_i(\nabla_XY)\Big\Vert^2,$$
and, from the fact that $\theta_i$ is nowhere zero, we obtain
$P_i(\nabla_XY)=0$ for any $i\in \{1,\dots,k\}$, which implies $\nabla_XY\in \Gamma(\mathcal D_0)$ for all $X\in \Gamma(TM)$, $Y\in \Gamma(\mathcal D_0)$; hence, the distribution $\mathcal D_0$ is completely integrable.

Also, for any $X\in \Gamma(TM)$ and $Y\in \Gamma(\oplus_{i=1}^k \mathcal D_i)$, we have $\nabla_XY\in \Gamma(\oplus_{i=1}^k \mathcal D_i)$ since
$$g(\nabla_XY,Z)=-g(Y,\nabla_XZ)=0$$
for all $Z\in \Gamma(\mathcal D_0)$. Therefore, we obtain (i).
\pagebreak

If the distribution $\mathcal D_i$ is integrable, then, for all $X,Y\in \Gamma(\mathcal D_i)$, we have $T[X,Y]\in \Gamma(\mathcal D_i)$, which implies
$$\nabla_XTY-\nabla_YTX=T(\nabla_XY-\nabla_YX)\in \Gamma(\mathcal D_i).$$
Conversely, if $\nabla_XTY-\nabla_YTX\in \Gamma(\mathcal D_i)$ for all $X,Y\in \Gamma(\mathcal D_i)$, then
$$T[X,Y]=\nabla_XTY-\nabla_YTX\in \Gamma(\mathcal D_i).$$
Applying $T$, we get $\sum_{j=0}^k\cos^2 \theta_j\cdot P_j[X,Y]\in \Gamma(\mathcal D_i)$ from Lemma \ref{le6} (ii), and, taking into account the orthogonality of the distributions and the fact that $\theta_j(p)\neq \frac{\pi}{2}$ for all $p\in M$ and $j\in \{0,\dots,k\}$ with $j\neq i$, we obtain (ii).

Now, since $\mathcal D_0$ is completely integrable, from the K\"{a}hler condition,
we deduce
$$n h(X,Y)-h(X,TY)=-N(\nabla_XY)=0$$
for all $X,Y\in \Gamma(\mathcal D_0)$.
Writing this relation for $TY$ instead of $Y$, we get
$$n^2 h(X,Y)=-h(X,Y)$$
for all $X,Y\in \Gamma(\mathcal D_0)$, and we obtain (iii).
\end{proof}

\begin{theorem}
Let $M$ be a pointwise $k$-slant submanifold of a K\"{a}hler manifold $(\bar M,\varphi,g)$.
If $\nabla N=0$, then:

(i) $M$ is a $(\mathcal D_0,\mathcal D_i)$-mixed totally geodesic submanifold of $\bar M$ (i.e., $h(X,Y)=\nolinebreak 0$ for all $X\in \Gamma(\mathcal D_0)$ and $Y\in \Gamma(\mathcal D_i)$) for $i\in \{1,\dots,k\}$;

(ii) for $i\in \{0,\dots,k\}$, either $M$ is a $(\mathcal D_i,\mathcal D_i)$-totally geodesic submanifold of $\bar M$, or $(-\cos^2\theta_i)$ is an eigenvalue function of $n^2$, and, for $X,Y\in \Gamma(\mathcal D_i)$, any nonzero $h(X,Y)$ is an eigenvector for it;

(iii) for $i\in\{0,\dots,k\}$, either $h(X,Y)=0$ for any $X\in \Gamma(\mathcal D_i)$ and
\linebreak
$Y\in \Gamma(TM)$, or ($-\cos^2\theta_i$) is an eigenvalue function of $T^2$, and, for $X\in \Gamma(\mathcal D_i)$ and $V\in \Gamma(T^{\bot}M)$, any nonzero $A_VX$ is an eigenvector for it.
\end{theorem}
\begin{proof}
Since $\nabla N=0$, we have $h(X,TY)=nh(X,Y)$
for all \linebreak
$X,Y\in \Gamma(TM)$, which implies
$$n^2h(X,Y)=-\cos^2\theta_i\cdot h(X,Y)$$
for all $Y\in \Gamma(\mathcal D_i)$. For $X,Y\in \Gamma(\mathcal D_i)$, we deduce (ii).
For $X\in \Gamma(\mathcal D_0)$ and $Y\in \Gamma(\mathcal D_i)$, we obtain
$$n^2h(X,Y)=n^2h(Y,X)=nh(Y,TX)=h(Y,T^2X)=-h(Y,X)=-h(X,Y),$$
and we get
$$\sin^2\theta_i\cdot h(X,Y)=0,$$
hence (i).

Applying $T$ to the relation from Proposition \ref{k} (ii)(2), we infer
$$T^2(A_VX)=A_V(T^2X)$$
for all $X\in \Gamma(TM)$ and $V\in \Gamma(T^{\bot}M)$; therefore, for all $X\in \Gamma(\mathcal D_i)$, we have
$$T^2(A_VX)=-\cos^2\theta_i\cdot A_VX,$$
and we get (iii).
\end{proof}

\begin{theorem}\label{h}
Let $M$ be a connected pointwise $k$-slant submanifold of an almost Hermitian manifold $(\bar M,\varphi,g)$. Then:

(i) $\nabla T^2=0$ if and only if $M$ is a $k$-slant submanifold, and $\nabla$ restricts to $\mathcal D_i$ (i.e., $\nabla_XY\in \Gamma(\mathcal D_i)$ for all $X\in \Gamma(TM)$ and $Y\in \Gamma(\mathcal D_i)$) for any $i\in \{0,\dots, k\}$;

(ii) for $i\in \{0,\dots, k\}$, $(\nabla_X T^2)Y=(\nabla_Y T^2)X$ for all $X,Y\in \Gamma(\mathcal D_i)$ if and only if $\mathcal D_i$ is a slant, integrable distribution. Furthermore, if $T^2$ is a Codazzi tensor field on $\mathcal D_i$ for all $i\in \{0,\dots, k\}$, then $M$ is a $k$-slant submanifold.
\end{theorem}
\begin{proof}
(i) Following the same steps as in \cite{blla}, we obtain
$$(\nabla_X T^2)Y= -\sum_{i=0}^kX(\cos^2 \theta_i)P_iY+\sum_{0\leq i,j \leq k}(\cos^2 \theta_i-\cos^2 \theta_j)P_i(\nabla_XP_jY)$$
for all $X,Y\in \Gamma(TM)$.
Taking into account the orthogonality of the distributions, the condition $\nabla T^2=0$ is equivalent to:
$$\sum_{j=0}^k(\cos^2 \theta_i-\cos^2 \theta_j)P_i(\nabla_XP_jY)-X(\cos^2 \theta_i)P_iY=0
$$
for all $X,Y\in \Gamma (TM)$ and any $i\in \{0,\dots, k\}$. We get
$$ X(\cos^2 \theta_i)Y=0
$$
for all $X\in \Gamma(TM)$ and $Y\in \Gamma(\mathcal D_i)$, $i\in \{0,\dots,k\}$, hence
$\theta_i$ is constant for all $i\in \{1,\dots,k\}$ (so $M$ is a $k$-slant submanifold). Also, $P_i(\nabla_XY)=0$ for all $Y\in \Gamma(\mathcal D_j)$, $i\neq j$;
hence, $\nabla$ restricts to $\mathcal D_j$ for any $j\in \{0,\dots,k\}$.
The converse implication follows immediately since $\nabla_XP_jY\in \Gamma(\mathcal D_j)$ for all $X,Y\in \Gamma(TM)$ and any $j\in \{0,\dots,k\}$.

(ii) On the other hand, for all $X,Y\in \Gamma(TM)$, we get
\begin{align*}
(\nabla_XT^2)Y-(\nabla_YT^2)X&= \sum_{j=0}^kP_j\Big(-X(\cos^2 \theta_j)Y+Y(\cos^2\theta_j)X\Big)\\
&\hspace{12pt}-\sum_{j=0}^kP_j\Big(\sum_{l=0}^k\cos^2\theta_l(\nabla_XP_lY-\nabla_YP_lX)\Big)\\
&\hspace{12pt}+\sum_{j=0}^kP_j\Big(\cos^2\theta_j(\nabla_XY-\nabla_YX)\Big).
\end{align*}
In particular, for $X,Y\in \Gamma(\mathcal D_i)$, we obtain
\begin{align*}
(\nabla_XT^2)Y-(\nabla_YT^2)X&= \Big(-X(\cos^2 \theta_i)Y+Y(\cos^2\theta_i)X\Big)\\
&\hspace{12pt}-\sum_{j=0}^kP_j\Big(\cos^2\theta_i(\nabla_XY-\nabla_YX)\Big)\\
&\hspace{12pt}+\sum_{j=0}^kP_j\Big(\cos^2\theta_j(\nabla_XY-\nabla_YX)\Big),
\end{align*}
and we deduce that $(\nabla_X T^2)Y=(\nabla_Y T^2)X$ for all $X,Y\in \Gamma(\mathcal D_i)$ if and only if
$$\left\{
   \begin{array}{ll}
     X(\cos^2\theta_i)Y-Y(\cos^2\theta_i)X=0 \\
     (\cos^2\theta_j -\cos^2\theta_i)P_j[X,Y]=0
   \end{array}
 \right.$$
for all $X,Y\in \Gamma(\mathcal D_i)$ and $j\neq i$. The first assertion is equivalent to the fact that $\theta_i$ is a constant, and the second one, since $\theta_i$ and $\theta_j$ are pointwise distinct for $i \neq j$, is equivalent to the integrability of $\mathcal D_i$.
\end{proof}

\smallskip

Hence, we recovered in Theorem \ref{h} (i) a known result proved by Chen \hspace{-2pt}\cite{geo}.

\smallskip

In particular, from the proof of Theorem \ref{h}, we deduce

\begin{corollary} \label{c8}
Let $M$ be a $k$-slant submanifold of an almost Hermitian manifold $(\bar M,\varphi,g)$. Then:

(i) for $i\in \{0,\dots,k\}$, $\mathcal D_i$ is completely integrable if and only if $(\nabla_XT^2)Y=0$ for all $X,Y\in \Gamma(\mathcal D_i)$;

(ii) for $i\in \{0,\dots,k\}$, $\mathcal D_i$ is integrable if and only if $(\nabla_X T^2)Y=(\nabla_Y T^2)X$ for all $X,Y\in \Gamma(\mathcal D_i)$.
\end{corollary}

%%% ENTER REFERENCES IN THE FORM

Adara M. Blaga\\
West University of Timi\c{s}oara\\
300223, Timi\c{s}oara, ROM\^{A}NIA\\
\emph{E-mail address}: {\tt adarablaga@yahoo.com}\\
ORCID ID: 0000-0003-0237-3866

\bigskip

Dan Radu La\c tcu\\
Central University Library of Timi\c{s}oara\\
300223, Timi\c{s}oara, ROM\^{A}NIA\\
\emph{E-mail address}: {\tt latcu07@yahoo.com}\\
ORCID ID: 0000-0003-1201-7400

\end{document}